\documentclass[a4paper,12pt]{article}

\usepackage{amsmath}
\usepackage{amsthm}
\usepackage{amssymb}
\usepackage{cancel}

\usepackage{epsfig}
\usepackage{psfrag}
\usepackage{subfigure}
\usepackage{graphicx}

\usepackage[mathscr,mathcal]{eucal}
\usepackage{enumerate, color}

\usepackage{t1enc}
\usepackage[latin2]{inputenc}

\def \R {\mathbb R}

\def \Z {\mathbb Z}

\def\cH{\mathcal{H}}

\def\cL{\mathcal{L}}

\def\cO{\mathcal{O}}

\def \eps {\epsilon}

\newcommand{\prob}[1]{\ensuremath{\mathbf{P}\left(#1\right)}}
\newcommand{\expect}[1]{\ensuremath{\mathbf{E}\left(#1\right)}}

\newcommand{\ind}[1]{\ensuremath{{1\!\!1}_{\{#1\}}}}

\def \Ordo {\cO}

\def\grad{\mathrm{grad}\,}
\def\div{\mathrm{div}\,}
\def\rot{\mathrm{rot}\,}
\def\curl{\mathrm{curl}\,}

\newcommand{\abs}[1]{\left|\,{#1}\,\right|}
\newcommand{\norm}[1]{\left\|\,{#1}\,\right\|}

\newtheorem {theorem}{Theorem}

\newtheorem {lemma}{Lemma}

\newtheorem* {theorem*}{Theorem}
\newtheorem* {thm*}{Theorem}
\newtheorem* {lemma*}{Lemma}
\newtheorem* {lem*}{Lemma}
\newtheorem* {corollary*}{Corollary}
\newtheorem* {cor*}{Corollary}
\newtheorem* {proposition*}{Proposition}
\newtheorem* {prop*}{Proposition}
\newtheorem* {definition*}{Definition}
\newtheorem* {def*}{Definition}
\newtheorem* {conjecture*}{Conjecture}
\newtheorem* {remark*}{Remark}
\newtheorem* {rem*}{Remark}

\theoremstyle{definition}

\newtheorem*{ack}{Acknowledgement}

\def\be{\begin{equation}}
\def\ee{\end{equation}}

\def\bea{\begin{eqnarray}}
\def\eea{\end{eqnarray}}

\newcommand{\wick}[1]{\ensuremath{:\!\! #1 \!\!:\,}}

\title{Superdiffusive bounds on self-repellent Brownian polymers and diffusion in the curl of the Gaussian free field in $d=2$}

\author{
{\sc B\'alint T\'oth} \qquad {\sc Benedek Valk\'o}
}

\begin{document}

\maketitle

\begin{abstract}

We consider two models of random diffusion in random environment in two dimensions. The first example is the \emph{self-repelling Brownian polymer}, this describes a diffusion pushed by the negative gradient of its own occupation time measure (local time). The second example is a \emph{diffusion in a fixed random environment  given by the curl of massless Gaussian free field}.

In both cases we show that the process is superdiffusive: the variance grows faster than linearly with time. We give lower and upper bounds of the order of $t \log \log t$, respectively, $t \log t$. We also present computations for an anisotropic version of the self-repelling Brownian polymer where we give lower and upper bounds of $t (\log t)^{1/2}$, respectively, $t \log t$. The bounds are given in the sense of Laplace transforms, the proofs rely on the resolvent method.

The true order of the variance for these processes is expected to be $t (\log t)^{1/2}$ for the isotropic and $t (\log t)^{2/3}$ for the non-isotropic case. In the appendix we present a non-rigorous  derivation of these scaling exponents.


\end{abstract}

\section{Introduction}
\label{s:intro}

We consider two models of random motion in random environment in $d=2$:
\begin{itemize}
\item
\emph{self-repelling Brownian polymer} process, abbreviated in the sequel as SRBP, respectively,

\item
\emph{diffusion in the curl of massless Gaussian free field}, abbreviated in the sequel as DCGF.
\end{itemize}
In both cases the critical dimension of the model-class is $d=2$: for $d\ge3$ the displacements are diffusive (i.e.~the variance grows linearly in time) and for $d=2$ multiplicative logarithmic corrections are expected. We provide a rigorous proof for the superdiffusivity: we give lower and upper bounds of order $t\log\log t$, respectively, $t\log t$ on the variance of the displacement in the sense of Laplace transforms. The lower bounds are the more interesting, the upper bound being almost straightforward.

The SRBP model is continuous space-time counterpart of the so called \emph{true self-avoiding random walk} (TSAW). The class of models has a long history. It first appeared in the theoretical physics literature where the models were formulated, and based on scaling and renormalization group arguments dimension dependent scaling behavior was conjectured. See \cite{amit_parisi_peliti_83}, \cite{obukhov_peliti_83}, \cite{peliti_pietronero_87}. Shortly later and seemingly totally independently of the physics papers mentioned, similar questions arose in the probabilistic literature. See   \cite{norris_rogers_williams_87}, \cite{durrett_rogers_92} where further conjectures and partial results appeared. For a concise historical account of the models, conjectures and results we refer the reader to the survey \cite{toth_01} and the introduction of the more recent paper \cite{horvath_toth_veto_10}. Our present result completes the picture in the sense that after the $d=1$ and $d\ge3$ behaviour being more-or-less  clarified in \cite{toth_95}, \cite{toth_werner_98}, \cite{toth_veto_09}, \cite{tarres_toth_valko_09}, respectively, in \cite{horvath_toth_veto_10}, we now  settle the question of  superdiffusivity in $d=2$. However, note that our bounds are still far from being sharp. There is plenty of room left for improvements.

The DCGF model belongs to the class of random walks and diffusions in random environment. The drift field, being the curl of a scalar field in 2d, is divergence free and hence the environment seen by the random walker is a priori stationary. There are too many papers on this topics to list here. An instructive and rich survey, though not the most recent one, is \cite{kozlov_1985}. For a more recent survey see chapter 11 of  \cite{komorowski_landim_olla_2011}. It turns out that the model considered here, where the drift field is the curl of 2d massless Gaussian free field (mollified by convolving with a smooth approximate delta-function) is just on the boundary between diffusive and superdiffusive asymptotics. The more robustly superdiffusive cases were considered in \cite{komorowski_olla_02}. Our  results complete the picture presented in \cite{komorowski_olla_02} in the sense of proving  superdiffusivity in this borderline case.

The expected order of the diffusivity for both of these models is $t (\log t)^{1/2}$ and it is conjectured that this is universal among a wide class of two-dimensional isotropic models. (See \cite{landim_ramirez_yau} for another example.)
In the appendix (Section \ref{s:appendix})  we present a \emph{non-rigorous}, nevertheless very instructive scaling argument which explains this conjecture. The argument dates back to the late sixties, cf. \cite{alder_wainwright_67}, \cite{alder_wainwright_70}, \cite{forster_nelson_stephen_70}. In our opinion this sheds sharp light on the origins of superdiffusivity in tracer diffusion models. The argument shows that the isotropy of the model is important: in particular for non-isotropic two-dimensional models the argument gives $t (\log t)^{2/3}$ for the diffusivity. This was rigorously proved for the diffusivity of a second class particle for finite range asymmetric exclusion models by Yau \cite{yau_04} building on the results of \cite{landim_quastel_salmhofer_yau_04}. To illustrate the difference we also consider a non-isotropic version of SRBP, for this model we give lower and upper bounds on the diffusivity  of order $t(\log t)^{1/2}$ respectively $t\log t$.

The structure of the paper is the following.
In the remaining parts of the Introduction we fix notation (subsection \ref{ss:general_notation}), formulate the models in rigorous mathematical terms (subsection \ref{ss:processes}), describe the picture of the environment as seen by the random walker (subsection \ref{ss:environment_process}) and formulate the main result of the paper (subsection \ref{ss:results}). In section \ref{s:hilbert_space_computations} the relevant Fock space and the relevant operators therein are presented (subsection \ref{ss:spop}) and the variational problem is explicitly formulated (subsection \ref{ss:variational_formula}). Section \ref{s:computations} contains the computational parts of the proof: we give upper bounds on the nontrivial term in the variational problem (subsection \ref{ss:upper_bound_on_J3}) and prove the main lemmas needed for completing the proof of the main result (subsection \ref{ss:bounds_on_D}).
Subsection \ref{ss:nonisotrop} contains the computational parts of the proof for the non-isotropic model.
Finally, section \ref{s:appendix} is an appendix which contains the mentioned non-rigorous scaling argument.

\subsection{Some general notation}
\label{ss:general_notation}

Throughout this paper we are in two dimensions. $V:\R^2\to\R$ will denote a (fixed) approximate delta-function, that is a smooth ($C^{\infty}$), spherically symmetric function with sufficiently fast decay at infinity (exponential decay certainly suffices). We also impose the condition of \emph{positive definiteness} of $V$:
\[
\hat V(p) = \int_{\R^2} e^{i p\cdot x}V(x) dx
\ge0.
\]
Occasionally we will also use the notation $U:\R^2\to\R$ for the unique positive definite function which yields
\[
V=U*U,
\qquad
\hat U(p)=\sqrt {\hat V(p)}.
\]
A particular choice could be
\[
V(x) = \frac{e^{-\abs{x}^2/(2\sigma^2)}}{2\pi\sigma^2},
\qquad
U(x) = \frac{e^{-\abs{x}^2/\sigma^2}}{\pi\sigma^2},
\]
but the choice of the concrete function $V$ is of no importance.

Partial derivatives (with respect to space coordinates) will be denoted $\partial_i$, $i=1,2$.  The \emph{gradient} and \emph{curl} of a smooth scalar field $A:\R^2\to\R$ are the vector fields
\[
\grad A  = (\partial_1 A, \partial_2 A),
\qquad
\curl A  = (\partial_2 A, - \partial_1 A).
\]
The \emph{divergence} and \emph{rotation} of a smooth vector field $A:\R^2\to\R^2$ are the scalar fields
\[
\div A(x)  = \partial_1 A_1+\partial_2 A_2,
\qquad
\rot A(x)  = \partial_2 A_1-\partial_1 A_2.
\]
Given a vectorial expression $b=(b_1,b_2)$ we will denote
\[
\tilde b=(\tilde b_1, \tilde b_2) = (b_2, -b_1).
\]
$t\mapsto B(t)$ will always denote a two-dimensional standard Brownian motion.

\subsection{The displacement processes considered}
\label{ss:processes}

\bigskip
\noindent
{\bf The self-repelling Brownian polymer process (SRBP)}

\medskip
\noindent
Let $x\mapsto F(x)$  be a smooth gradient (that is rotation free) vector field on $\R^2$ with slow increase at infinity. We define the stochastic process $t\mapsto X(t)\in\R^2$ as the solution of the following SDE:
\begin{equation}
\label{srbp_diff}
d X(t)=
\left(F(X(t))-\int_0^t \grad V (X(t)-X(u))d u\right) dt + \sqrt{2}\,dB(t).
\end{equation}
Introducing the \emph{occupation time measure} (also called \emph{local time} in this paper)
\[
\ell(t, A) = \abs{\{0<s\le t: X(s)\in A\}}
\]
where $A\subset \R^2$ is any measurable domain, we can rewrite the previous SDE as follows:
\begin{equation}
\label{srbp_diff2}
dX(t)=
\big(F - \grad V*\ell(t,\cdot) \big) (X(t))dt + \sqrt{2}\, dB(t).
\end{equation}
The form \eqref{srbp_diff2} of the driving mechanism, shows explicitly the phenomenological meaning of the law of the process: it is pushed by the negative gradient of its own local time towards domains less visited in the past.

\bigskip
\noindent
{\bf The non-isotropic SRBP}

\noindent
Let now $F_1:\R^2\to\R$  be a smooth scalar field on $\R^2$ with slow increase at infinity.
The process is again defined as a solution of an SDE similar to (\ref{srbp_diff}), but instead of $\grad$ we use the operator $(\partial_1, 0)$:
\[
dX_1(t)=\left(F_1(X(t))-\int_0^t \partial_1 V(X(t)-X(u))d u\right) dt + \sqrt{2}\, dB_1(t),
\qquad
dX_2(t)= \sqrt{2}\, dB_2(t).
\]
Or, expressed in terms of gradient of local time:
\[
dX_1(t)=
\big(F_1 - \partial_1 V*\ell(t,\cdot) \big) (X(t))dt + \sqrt{2}\, dB_1(t),
\qquad
dX_2(t)= \sqrt{2}\, dB_2(t).
\]
Now, only the first coordinate of the displacement is pushed by the corresponding negative gradient component of mollified local time. The second coordinate moves as memoryless Brownian motion.

\bigskip
\noindent
{\bf Diffusion in the curl of Gaussian free field (DCGF)}

\medskip
\noindent
Now let $x\mapsto F(x)$ be a smooth curl (that is divergence free) vector field on $\R^2$ with slow increase at infinity. We define the stochastic process $t\mapsto X(t)\in\R^2$ as the solution of the SDE
\begin{equation}
\label{dcre_diff}
d X(t)=
F(X(t)) dt + \sqrt{2}\, dB(t).
\end{equation}
This is a diffusion in the drift field $F$.

\subsection{The environment seen from the moving point}
\label{ss:environment_process}

\bigskip
\noindent
{\bf SRBP}

\medskip
\noindent
The environment profile appearing on the right-hand side of \eqref{srbp_diff2}, as seen in a moving coordinate frame tied to the current position of the displacement process is $t\mapsto\eta(t,\cdot)$:
\[
\eta(t,x)
 =
\big(F - \grad V*\ell(t,\cdot) \big) (X(t)+x)
\]
with initial value
\[
\eta(0,x)
 =
F(x).
\]
$t\mapsto\eta(t,\cdot)$ is a Markov process with continuous sample paths in the Fr\'echet space
\[
\Omega  =  \big\{\omega\in C^{\infty}(\R^2\to\R^2)\,:\,
\rot\omega\equiv0,\,\,\,
\norm{\omega}_{k,m,r}<\infty \big\}
\]
where $\norm{\omega}_{k, m,r}$ are the seminorms
\begin{equation}
\label{seminorms}
\norm{\omega}_{k,m,r}  =  \sup_{x\in\R^d} \,
\big(1+\abs{x}\big)^{-1/r} \, \abs{\partial^{\abs{m}}_{m_1,\dots,m_d}\omega_k(x)}
\end{equation}
defined for $k=1,2$, multiindices $m=(m_1,\dots,m_d)$, $m_j\ge0$, and $r\ge1$.

Let $\pi$ be the Gaussian measure on $\Omega$ defined by the covariances
\[
\int_\Omega \omega_k(x) d\pi(\omega)=0,
\qquad
K_{kl}(x-y)
=
\int_\Omega \omega_k(x) \omega_l(y) d\pi(\omega)=
V*g_{kl}(x-y)
\]
where
\[
g_{kl}(x)
=
-\partial^2_{kl} \log \abs{x}
=
\frac{\tilde x_k \tilde x_l}{\abs{x}^{-3}}.
\]
The Fourier transform of the correlations is
\[
\hat K_{kl}(p)
=
\frac{p_kp_l}{\abs{p}^{2}} \hat V(p).
\]
It is clear that $\omega$ distributed according to $\pi$ \emph{the gradient of the massless Gaussian free field smeared out by convolution with $U$}.

It has been proved in \cite{horvath_toth_veto_10} that the Gaussian probability measure $\pi$ is stationary and ergodic for the Markov process $\eta_t$. That is: if the initial vector field $F$ is sampled from this distribution, then the vector field profile seen from the position of the moving particle will have the same distribution at any later time. (Although \cite{horvath_toth_veto_10} deals with the $d\ge 3$ case the same proof applies for $d=2$. See also \cite{tarres_toth_valko_09} for the 1d case.) The process will be considered in this stationary regime. That is, the  initial profile $\eta_0=F$ is distributed according to the stationary measure $\pi$.

\bigskip
\noindent
{\bf Non-isotropic SRBP}

\medskip
\noindent
The situation is similar to the previous case, but our environmental profile will be the scalar field  $\xi:\R^2\to \R$
\[
\xi(t,x)=\big(F - \partial_1 V*\ell(t,\cdot) \big) (X(t)+x).
\]
This will be now a Markov process with continuous sample paths on
\[
\Omega  =
\big\{\omega\in C^{\infty}(\R^2\to\R)\,:\,
\norm{\omega}_{m,r}<\infty \big\}
\]
and the seminorms $\norm{\omega}_{m,r}$ are defined very similarly to \eqref{seminorms}.

Computations very similar to those in \cite{tarres_toth_valko_09} and \cite{horvath_toth_veto_10} show that the Gaussian measure $\pi$ defined by the expectations and covariances
\[
\int_\Omega \omega(x) d\pi(\omega)=0,
\qquad
K(x-y)
=
\int_\Omega \omega(x) \omega(y) d\pi(\omega)=V(x-y)
\]
is stationary and ergodic for this Markov process.

In order to keep unified and consistent notation, in the anisotropic case we will still use vectorial notation for the scalar field $\omega$, as follows:
\[
\omega_1(x)=\omega(x), \qquad \omega_2(x)\equiv0.
\]
In this case the probability measure $\pi$ will be denoted as Gaussian distribution of a two component vector field with covariances
\[
K_{kl}(x-y)=V(x-y) \delta_{k,1}\delta_{l,1}.
\]

\bigskip
\noindent
{\bf DCGF}

\medskip
\noindent
The environment seen by the moving point is now $t\mapsto\eta(t,\cdot)$ defined as:
\[
\eta(t,x)
 =
F(X(t)+x).
\]
Now, $t\mapsto\eta(t,\cdot)$ is again a Markov process with continuous sample paths in the Fr\'echet space
\[
\Omega  =  \big\{\omega\in C^{\infty}(\R^2\to\R^2)\,:\,
\div\omega\equiv0,\,\,\,
\norm{\omega}_{k,m,r}<\infty \big\}
\]
where the seminorms $\norm{\omega}_{k, m,r}$ are formally defined as in \eqref{seminorms}.

Let now $\pi$ be  the Gaussian probability measure on $\Omega$ defined by the expectations and covariances
\[
\int_\Omega \omega_k(x) d\pi(\omega)=0,
\qquad
K_{kl}(x-y):=
\int_\Omega \omega_k(x)\omega_l(y) d\pi(\omega)
=
V*g_{kl}(x-y),
\]
where now
\[
g_{kl}(x)
=
-\tilde\partial^2_{kl} \log \abs{x}
=
\frac{x_k  x_l}{\abs{x}^{-3}}.
\]
The Fourier transform of the correlations is now
\[
\hat K_{kl}(p)=\frac{\tilde p_k\tilde p_l}{\abs{p}^{2}} \hat V(p).
\]
The random vector field $\omega_{kl}(x)$ is now \emph{the curl of the massless Gaussian free field smeared out by convolution with $U$}.

It is a well known fact that the the process $\eta(t)$ of a diffusion \eqref{dcre_diff} in any translation invariant and ergodic \emph{divergence free} vector field $F$ is (time) stationary and ergodic (see e.g. \cite{kozlov_1985}, or chapter 11 of \cite{komorowski_landim_olla_2011}).  In particular, this holds if the drift field $F$ is sampled from the distribution $\pi$.

\subsection{Superdiffusive bounds}
\label{ss:results}

For both processes considered we denote
\[
E(t) = \expect{\abs{X(t)}^2},
\qquad
\hat E(\lambda) = \int_0^\infty E(s)e^{-\lambda s} ds.
\]

The main result of the present paper is the superdiffusive lower bound stated in the following theorem, which is valid for both processes.

\begin{theorem}
\label{thm:main}
Let $X(t)$ be either one of the (isotropic) processes SRBP or DCGF started with initial environment profile sampled from the respective stationary distributions. There exist constants $0< C_1,C_2 <\infty$ such that for $0<\lambda<1$ the following bounds hold
\begin{equation}
\label{diffusivity_bounds}
C_1\lambda^{-2}\log\abs{\log\lambda}
\le  \hat E(\lambda) \le
C_2 \lambda^{-2}\abs{\log\lambda}.
\end{equation}
\end{theorem}

\noindent
{\bf Remarks:}
(1)
Modulo Tauberian inversion, these bounds mean in real time
\[
C_3 t \log\log t
\le \expect{\abs{X(t)}^2} \le
C_4 t \log t,
\]
with $0< C_3,C_4 <\infty$ and for $t$ sufficiently large.
\\
(2)
Based on the Alder-Wainwright argument sketched in the Appendix (Section \ref{s:appendix}) the expected true order is $\expect{\abs{X(t)}^2}\asymp t(\log t)^{1/2}$.
\bigskip

The proof of Theorem \ref{thm:main} follows the main lines of \cite{landim_quastel_salmhofer_yau_04}. However on the computational level there are some notable differences. The bounds exploited in \cite{landim_quastel_salmhofer_yau_04} which rely inter alia on a clever application of Schwarz's inequality don't yield superdiffusive lower bounds in our case. This is due to the isotropy of our model as opposed to the anisotropy of the asymmetric simple exclusion models. From the Alder-Wainwright scaling argument it follows (see the Appendix) that for anisotropic models $\expect{\abs{X(t)}^2}\asymp t(\log t)^{2/3}$, while for isotropic models $\expect{\abs{X(t)}^2}\asymp t(\log t)^{1/2}$ is naturally expected. Thus, the isotropic models are less superdiffusive than the anisotropic ones. This phenomenon manifests also on computational level: proof of superdiffusive lower bound for isotropic models exhibits some subtle a priori extra difficulty.

For the non-isotropic model we have the following bounds:

\begin{theorem}
\label{thm:masodik}
If $X(t)$ is the non-isotropic SRBP then we get the bounds
\begin{equation}
\label{diffusivity_bounds_1}
C_1\lambda^{-2}\abs{\log\lambda}^{1/2}
\le \hat E(\lambda) \le
C_2 \lambda^{-2}\abs{\log\lambda}.
\end{equation}
\end{theorem}

\noindent
{\bf Remark:}
Modulo Tauberian inversion, these bounds mean in real time
\[
C_3 t (\log t)^{1/2}
\le \expect{\abs{X(t)}^2} \le
C_4 t \log t,
\]
with $0< C_3,C_4 <\infty$ and for $t$ sufficiently large.

\section{Fock space computations}
\label{s:hilbert_space_computations}

\subsection{Spaces and operators}
\label{ss:spop}

For basics of Gaussian Hilbert spaces (that is: Fock spaces) see \cite{simon_74} or \cite{janson_97}. Let $(\Omega,\pi)$ be either one of the probability spaces identified in subsection \ref{ss:environment_process}. The Gaussian Hilbert space $\cH = \cL^2(\Omega, \pi)$ is naturally graded
\[
\cH=\bigoplus_{n=0}^\infty \cH_n
\]
where
\[
\cH_n  =  \mathrm{span}\{\wick{\omega_{l_1}(x_1)\dots\omega_{l_n}(x_n)} \,:\, l_1,\dots,l_n=1,2, \,\,\, x_1,\dots, x_n\in\R^2\}.
\]
Here and in the sequel $\wick{Z_1\dots Z_n}$ denotes the Wick monomial formed by the jointly Gaussian random variables $Z_1,\dots,Z_n$.

The operators $\nabla_k$, $k=1,2$,  and $\Delta$ are the usual ones. Their action on Wick monomials is
\begin{align*}
&
\nabla_k
\wick{\omega_{l_1}(x_1)\dots\omega_{l_n}(x_n)}
 =
\sum_{m=1}^n
\wick{\omega_{l_1}(y_1)\dots\partial_k\omega_{l_m}(x_m)\dots\omega_{l_n}(y_n)}, \\
&
\Delta
 =
\sum_{k=1}^2 \nabla_k^2.
\end{align*}
They are defined as unbounded operators on $\cH$ by graph closure.
We will also need the following creation and annihilation operators. Again, we give their action on Wick monomials:
\begin{align*}
&
a^*_{k}
\wick{\omega_{l_1}(x_1) \dots \omega_{l_n}(x_n)}
 =
\wick{\omega_k(0) \omega_{l_1}(x_1) \dots \omega_{l_n}(x_n)}
\\
&
a_{k}
\wick{ \omega_{l_1}(x_1) \dots \omega_{l_n}(x_n) }
 =
\sum_{m=1}^n
K_{k l_m}(x_m)
\wick{\omega_{l_1}(y_1) \dots \cancel{\omega_{l_m}(x_m)} \dots \omega_{l_n}(y_n)}.
\end{align*}
Note that the creation and annihilation operators defined in the context of the different models models are \emph{not} unitary equivalent.

In both isotropic cases (SRBP and DCGF) the infinitesimal generators of the semigroups of the stationary Markov process $\eta_t$, acting on $\cL^2(\Omega,\pi)$, is expressed as
\[
G
=
-S+A_++A_-,
\]
\[
S=-\Delta,
\qquad
A_+=\sum_{k=1}^2 a^*_k\nabla_k,
\qquad
A_-=\sum_{k=1}^2 \nabla_k a_k.
\]
This follows from standard computations in the case of DCGF. For the SRBP process it relies on somewhat more complex considerations with some involvement of Malliavin calculus. For details see \cite{horvath_toth_veto_10}.
Note, that the infinitesimal generators of the two processes, although formally similarly expressed, are \emph{not} unitary equivalent.

For the non-isotropic SRBP, since $\omega_2(x)\equiv0$, $a_2^* = a_2 =0$,  the asymmetric part of the generator is slightly modified:
\[
A_+= a^*_1\nabla_1,
\quad
A_-=\nabla_1 a_1.
\]

\subsection{The variational formula}
\label{ss:variational_formula}

We write the displacement as sum of a martingale and compensator term. In each of the cases we will concentrate on the first coordinate component.
\begin{equation}
\label{martingale+compensator_}
X_1(t)=B_1(t)+\int_0^t\varphi(\eta_s) ds,
\end{equation}
where
\[
\varphi:\Omega\to \R,
\qquad
\varphi(\omega)=\omega_1(0).
\]
Note, that actually $\varphi\in\cH_1$.

The martingale term is diffusive, so in order to prove superdiffusive bounds we have to focus on the integral on the right hand side of \eqref{martingale+compensator_}. Laplace transformation yields:
\begin{equation}
\label{laplace_transform}
\int_0^\infty e^{-\lambda t}
\expect{\left(\int_0^t\varphi(\eta_s) ds\right)^2} dt
=
\lambda^{-2}(\varphi, R_\lambda \varphi),
\end{equation}
where
\[
R_\lambda=\int_0^\infty e^{-\lambda t} e^{t G} = (\lambda I - G)^{-1},
\]
is the resolvent of the infinitesimal generator $G$.
We are going to prove bounds for the right hand side of \eqref{laplace_transform}.

The following variational formula is straightforward, see e.g. \cite{landim_quastel_salmhofer_yau_04}:
\[
(\varphi, R_\lambda \varphi)
=
\sup_{\psi\in\cH}
\big\{
2(\varphi, \psi) - (\psi, (\lambda+S)\psi) - (A\psi, (\lambda+S)^{-1} A\psi)
\big\}
\]

The upper bounds in \eqref{diffusivity_bounds}, \eqref{diffusivity_bounds_1} drop out essentially for free:
\[
(\varphi, R_\lambda \varphi)
\le
\sup_{\psi\in\cH}
\big\{2(\varphi, \psi) - (\psi, (\lambda+S)\psi)\big\}
=
(\varphi, (\lambda I-\Delta)^{-1}\varphi)
\le
C\log\lambda.
\]
The last bound follows from straightforward computations on so-called diffusion in random scenery, and proves the upper bound in both Theorem \ref{thm:main} and \ref{thm:masodik}.

In order to get the more interesting lower bound we write first
\begin{align}
\notag
(\varphi, R_\lambda \varphi)
&\ge
\sup_{\psi\in\cH_1}
\big\{
2(\varphi, \psi) - (\psi, (\lambda+S)\psi) - (A\psi, (\lambda+S)^{-1} A\psi)
\big\}
\\
\label{lower_bound}
&=
\sup_{\psi\in\cH_1}
\big\{
2(\varphi, \psi) - (\psi, (\lambda+S)\psi) - (A_+\psi, (\lambda+S)^{-1} A_+\psi)
\big\}.
\end{align}
Note that for any $\psi\in\cH_1$, $A_-\psi=0$.

We will treat the isotropic SRBP and DCGF processes and the anisotropic SRBP separately.

\bigskip
\noindent
{\bf Isotropic SRBP and DCGF models:}

\medskip
\noindent
We write the variational test function $\psi\in\cH_1$ as
\[
\psi(\omega)=\sum_{l=1}^2 \int_{\R^2} u_l(x)\omega_l(x) dx,
\]
where $u:\R^2\to\R^2$ is such that, $u(-x)=u(x)$, (and thus $\hat u:\R^2\to\R^2)$, and
\begin{align*}
\mathrm{SRBP:}
&&&&&&&
p\times \hat u(p) \equiv 0,
\qquad
&&
\int_{\R^2}
\hat V(p) \left(p\cdot \hat u(p)\right)^2 dp<\infty,
\\
\mathrm{DCGF:}\hfill\hfill
&&&&&&&
p\cdot \hat u(p) \equiv 0,
\qquad
&&
\int_{\R^2}
\hat V(p) \left(p\times \hat u(p)\right)^2 dp<\infty.
\end{align*}
We define $v:\R^2\to i\,\R$ be
\begin{align*}
\mathrm{SRBP:}
&&&&&&&
v=i \, \div u, \qquad
&&
\hat v(p) = p\cdot \hat u(p),
\\
\mathrm{DCGF:}
&&&&&&&
v=i \, \rot u, \qquad
&&
\hat v(p) = p\times \hat u(p).
\end{align*}
Then, in both cases $\hat v:\R^2\to\R$ and
\begin{equation}
\label{conditions_for_hat_v}
\hat v(-p)=-\hat v(p),
\quad\text{ and }
\quad
\int_{\R^2}
\frac{\hat V(p)}{\abs{p}^2}\hat v (p)^2 dp <\infty.
\end{equation}

The first two  terms on the right hand side of the variational lower bound \eqref{lower_bound} are computed directly. \begin{align}
\label{first_srbp}
\mathrm{SRBP:}
&&&&&&&
(\varphi,\psi)
=
\int_{\R^2}
\hat V(p)
\frac{p_1}{\abs{p}^2}  \hat v(p) dp
=:J_1(\hat v)
\\
\label{first_dcgf}
\mathrm{DCGF:}
&&&&&&&
(\varphi,\psi)
=
\int_{\R^2}
\hat V(p)
\frac{p_2}{\abs{p}^2}  \hat v(p) dp
=:J_1(\hat v)
\\
\label{second_sipdre}
\mathrm{SRBP \& DCGF:}
&&&&&&&
(\psi, (\lambda + S) \psi)
=
\int_{\R^2}
\hat V(p)
\frac{\lambda+\abs{p}^2}{\abs{p}^2} \hat v(p)^2 dp
=:J_2(\hat v).
\end{align}
In order to compute the third term on the right hand side of \eqref{lower_bound} we first consider the SRBP case. We have
\[
A_+\psi=
-\sum_{k,l=1}^2\int_{\R^2}\partial_ku_l(x)\wick{\omega_k(0)\omega_l(x)} dx,
\]
and hence
\begin{align}\label{harmadik}
&(A_+\psi, (\lambda+S)^{-1}A_+\psi)=
\\
&\hskip10mm
\sum_{k,l=1}^2\sum_{m,n=1}^2\int_{\R^2}\int_{\R^2}\int_{\R^2}
\partial_mu_n(y)\partial_ku_l(x) g_\lambda(z) \expect{\wick{\omega_m(0)\omega_n(y)}\wick{\omega_k(z)\omega_l(x+z)}} dx dy dz. \nonumber
\end{align}
Here $g_\lambda(z)$ is the integral kernel of the operator $(\lambda-\Delta)^{-1}$ in $\R^2$, with Fourier transform
\[
\hat g_\lambda (p)= \frac{1}{\lambda + \abs{p}^2}.
\]
For the DCGF case one finds very similar formulas, with $\partial_k$ replaced by $\tilde \partial_k$.

The expectation inside the integral in \eqref{harmadik} is computed by exploiting the fact that the fields $\omega$ are Gaussian and hence the four point functions arising  are expressed in terms of the covariances. After some straightforward computations, eventually we get:
\begin{align}
\label{third_srbp}
\mathrm{SRBP:}
&&&&&&&
(A_+\psi, (\lambda+S)^{-1} A_+\psi)
=
\\
\notag
&&&&&&&\qquad
\int_{\R^2}\int_{\R^2}
\hat V(p) \hat V(q)
\frac{(p\cdot q)^2 }{\abs{p}^2\abs{q}^2} \frac{1}{\lambda+\abs{p-q}^2}
\left(\hat v(p)-\hat v(q)\right)^2 dqdp
=:J_3 (\hat v),
\\[15pt]
\notag
\mathrm{DCGF:}
&&&&&&&
(A_+\psi, (\lambda+S)^{-1} A_+\psi)
=
\\
\notag
&&&&&&&\qquad
\int_{\R^2}\int_{\R^2}
\hat V(p) \hat V(q)
\frac{(p\times q)^2 }{\abs{p}^2\abs{q}^2} \frac{1}{\lambda+\abs{p-q}^2}
\left(\hat v(p)-\hat v(q)\right)^2 dqdp
=:J_3(\hat v).
\end{align}
And the variational problem \eqref{lower_bound} becomes
\begin{equation}
\label{variational_problem}
(\varphi, R_\lambda \varphi)
\ge
\sup_{\hat v} \big(2J_1(\hat v) - J_2 (\hat v)- J_3(\hat v)\big),
\end{equation}
with supremum taken over functions $\hat v:\R^2\to\R$ satisfying conditions \eqref{conditions_for_hat_v}.

Given the explicit expressions \eqref{first_srbp}, \eqref{first_dcgf} and \eqref{second_sipdre} for $J_1(\hat v)$ and $J_2(\hat v)$, in order to prove the superdiffusive lower bound in \eqref{diffusivity_bounds} we need to prove efficient upper bounds on $J_3(\hat v)$.\bigskip

\bigskip
\noindent
{\bf Anisotropic SBRP case:}

\medskip
\noindent
The test function $\psi\in\cH_1$ is written in the form
\[
\psi(\omega)=\int_{\R^2} u(x) \omega(x) dx
 \]
with $u:\R^2\to \R$, $u(x)=u(-x)$. Then
\begin{align*}
&
(\varphi,\psi)=\int_{\R^2} \hat u(p) \hat V(p) d p
=:J_1(\hat u),
\\
&
(\psi, (\lambda + S) \psi)\notag
=\int_{\R^2}\hat V(p) ({\lambda+\abs{p}^2})\hat u(p)^2 dp
=:J_2(\hat u)
\\
&
(A_+\psi, (\lambda+S)^{-1} A_+\psi)
=\int_{\R^2}\int_{\R^2}
\hat V(p) \hat V(q) (p_1 \hat u(p)-q_1 \hat u(q))^2 dp dq
=:J_3(\hat u).
\end{align*}

\section{Computations}
\label{s:computations}

The following three subsections contain the computational parts of the proofs for the isotropic cases (subsections \ref{ss:upper_bound_on_J3} and \ref{ss:bounds_on_D}), respectively the anisotropic case (subsection \ref{ss:nonisotrop}).

\subsection{Upper bound on $J_3$, isotropic models}
\label{ss:upper_bound_on_J3}

Since the DCGF model is somewhat simpler than the SRBP, we will treat them in this order.

\bigskip
\noindent
{\bf DCGF:}

\medskip
\noindent
Denote
\[
D(\lambda, \abs{p})
=
4
\int_{\R^2}
\hat V(q)
\frac{(p\times q)^2 }{\abs{p}^2\abs{q}^2} \frac{1}{\lambda+\abs{p-q}^2}
dq.
\]
When estimating the function $D(\lambda,\abs{p})$ (see  Lemma \ref{lem:dcgf_D_bound} below and its proof) the vectorial product $p\times q$ will help substantially. As we shall see later this is not the case for the SRBP model.

By Schwarz's inequality we get
\[
J_3(\hat v)
\le
\int_{\R^2}
\hat V(p) D(\lambda, \abs{p}) \hat v(p)^2 dp.
\]
and the variational problem \eqref{variational_problem} is readily solved by
\[
\hat v^*(p)
=
\frac{p_2}{\lambda + \left(1+D(\lambda, \abs{p})\right)\abs{p}^2},
\]
which eventually yields the lower bound
\begin{equation}
\label{lower_bound_for_DCGF}
(\varphi, R_\lambda \varphi)
\ge
\int_{\R^2}
\hat V(p)
\big(\lambda + \left(1+D(\lambda, \abs{p})\right)\abs{p}^2\big)^{-1} dp.
\end{equation}
Thus, a suitable \emph{upper bound on} $D(\lambda, \abs{p})$ will provide a good \emph{lower bound on} $(\varphi,R_\lambda\varphi)$.

In subsection \ref{ss:bounds_on_D} we will prove the following

\begin{lemma}
\label{lem:dcgf_D_bound}
In the DCGF case, for $\lambda<1$ and $\abs{p}\le 1$ we have
\begin{equation}
\label{dcgf_D_bound}
D(\lambda, \abs{p})
\le
C \abs{\log(\lambda + \abs{p}^2)}.
\end{equation}
\end{lemma}

\noindent
Inserting \eqref{dcgf_D_bound} into \eqref{lower_bound_for_DCGF} we readily obtain
\[
(\varphi, R_\lambda \varphi)
\ge
\int_{\R^2} \ind{\abs{p}<1}
\hat V(p)
\big(\lambda + \left(1+ C \abs{\log(\lambda + \abs{p}^2)}\right)\abs{p}^2\big)^{-1} dp > C\log\abs{\log\lambda}.
\]
which proves the lower bound in Theorem \ref{thm:main} for the DCGF.

\bigskip
\noindent
{\bf SRBP:}

\medskip
\noindent
Applying a similar argument to the SRBP process will not yield super diffusive lower bound: the bound corresponding to \eqref{dcgf_D_bound} would be
\[
D(\lambda, \abs{p}) \le C \abs{\log \lambda}
\]
and this is simply not sufficient for superdiffusivity.

In order to get the true superdiffusive lower bound we have to split the integral on the right hand side of \eqref{third_srbp} according whether $\abs{p-q}$ is small or large and apply Schwarz's inequality only in the latter part. Let
\begin{align*}
&
J_{31}(\hat v):=
\int_{\R^2}\int_{\R^2}
\hat V(p) \hat V(q)
\frac{1}{\lambda+\abs{p-q}^2}
\left(\hat v(p)-\hat v(q)\right)^2
\ind{\abs{p-q}\ge \abs{p}/3}
dqdp
\\
&
J_{32}(\hat v)=
\int_{\R^2}\int_{\R^2}
\hat V(p) \hat V(q)
\frac{1}{\lambda+\abs{p-q}^2}
\left(\hat v(p)-\hat v(q)\right)^2
\ind{\abs{p-q}\le \abs{p}/3}
dqdp
\\
&
J_3(\hat v) \le J_{31}(\hat v)+J_{32}(\hat v).
\end{align*}
Then, by applying Schwarz's inequality
\begin{align*}
J_{31}(\hat v)
&
\le
2
\int_{\R^2}\int_{\R^2}
\hat V(p) \hat V(q)
\frac{1}{\lambda+\abs{p-q}^2}
\left(\hat v(p)^2+\hat v(q)^2\right)
\ind{\abs{p-q}\le \abs{p}/3}
dqdp
\\
&
\le
\int_{\R^2}
\hat V(p)
D(\lambda, \abs{p}) \hat v(p)^2
dp,
\end{align*}
where now
\[
D(\lambda, \abs{p})=
4
\int_{\R^2}
\hat V(q) \frac{1}{\lambda + \abs{p-q}^2} \ind{\abs{p-q}\ge \abs{p}/3} dq.
\]
In subsection \ref{ss:bounds_on_D} we prove the following

\begin{lemma}
\label{lem:srbp_D_bound}
In the isotropic SRBP case, for $\lambda<1$ and $\abs{p}\le 1$ we have
\[
D(\lambda, \abs{p})
\le
C \abs{\log(\lambda + \abs{p}^2/9)}.
\]
\end{lemma}
Now,  we give upper bound on $J_{32}$:
\begin{align*}
J_{32}
&
\le
\int_{\R^2}\int_{\R^2}
\hat V(p) \hat V(q)
\frac{1}{\lambda+\abs{p-q}^2}
\left( (p-q) \cdot \nabla \hat v (r(p,q))\right)^2
\ind{\abs{p-q}\le \abs{p}/3}
dqdp
\\
&
\le
\frac14
\int_{\R^2}
\hat V(p) \abs{p}^2 \sup_{r:\abs{r-p}<\abs{p}/3} \abs{\nabla\hat v(r)}^2 dp=: J'_{32}(\hat v)
\end{align*}
We choose the variational function
\begin{equation}
\label{choice}
\hat v^*(p):= c p_1 h(\lambda+|p|^2),
\qquad
h(x):=\frac{1}{x \log(x+x^{-1})}.
\end{equation}
Then at one hand, for $c$ sufficiently small we get
\begin{equation}
\label{srbp_upper_bound1}
J_1(\hat v^*) - J_2(\hat v^*) - J_{31}(\hat v^*) \ge C \log\abs{\log \lambda}.
\end{equation}
On the other hand, in the next subsection we also prove

\begin{lemma}
\label{lem:harmadik}
With the choice \eqref{choice} of the variational function we have
\begin{equation}
\label{srbp_bound_on_J32}
J_{32}'(\hat v^*)\le C.
\end{equation}
\end{lemma}

Finally, from \eqref{srbp_upper_bound1} and \eqref{srbp_bound_on_J32} it follows that
\[
(\varphi,R_\lambda\varphi)> C\log\abs{\log \lambda}
\]
which gives the lower bound  Theorem \ref{thm:main} for the SRBP.

\subsection{Proof of Lemma \ref{lem:dcgf_D_bound}, \ref{lem:srbp_D_bound} and \ref{lem:harmadik}}
\label{ss:bounds_on_D}

\begin{proof}[Proof of Lemma \ref{lem:dcgf_D_bound}]
By the rotational symmetry we may assume that $p>0$ is real. Using the decay of $\hat V$ the integral on $|q|>2$ is bounded by a fixed constant. Since $\hat V$ is bounded it is enough to bound
\begin{equation}
\label{egy}
\int_{|q|\le 2}
\frac{(p\times q)^2 }{\abs{p}^2\abs{q}^2} \frac{1}{\lambda+\abs{p-q}^2}
dq.
\end{equation}
Let $q=(r \cos(t), r \sin(t))$ then we may rewrite \eqref{egy} as
\begin{equation}
\label{ketto}
\int_{0}^2 \int_0^{2\pi}
\frac{p^2 r^2 \sin^2(t) }{p^2 r^2} \frac{1}{\lambda+\abs{p-r e^{it}}^2} r dt dr=\int_{0}^2 \int_0^{2\pi}
\frac{\sin^2(t)}{\lambda+r^2+p^2-2 p r \cos(t)} r dt dr.
\end{equation}
If $A\ge B\ge 0$ then
\[
\frac{\sin^2(t)}{A+B \cos(t)}+\frac{\sin^2(\pi+t)}{A+B \cos(\pi+t)}=\frac{2 A\sin^2(t)}{A^2-B^2 \cos^2(t)}\le \frac{2}{A}
\]
which shows that after integrating in $t$ in \eqref{ketto}:
\[
\int_{0}^2 \int_0^{2\pi}
\frac{\sin^2(t)}{\lambda+r^2+p^2-2 p r \cos(t)} r dt dr\le \int_{0}^2
\frac{2\pi}{\lambda+r^2+p^2} r dr=\log(\lambda+p^2+4)-\log(\lambda+p^2).
\]
Collecting all our estimates and using $\lambda<1, |p|\le 1$ the statement of the lemma follows.
\end{proof}

\begin{proof}[Proof of Lemma \ref{lem:srbp_D_bound}]
By the decay of $\hat V$ the integral on $|p-q|>2$ is  bounded by a fixed constant. Thus we only need to bound
\[
\int_{|p-q|\le 2}
\hat V(q) \frac{1}{\lambda + \abs{p-q}^2} \ind{\abs{p-q}\ge \abs{p}/3} dq=\int_{|p|/3}^2 \frac{1}{\lambda+x^2} xdx.
\]
The last integral is $\log(\lambda+2)-\log(\lambda+|p|^2/9)$ from which the lemma follows.
\end{proof}

\begin{proof}[Proof of Lemma \ref{lem:harmadik}]
Recall the definition of the variational function $\hat v^*$ from \eqref{choice}. In order to avoid heavy notation we will drop the $*$ from $\hat v^*$ in the subsequent computations. We have
\[
\nabla \hat v(p) = \big(h(\lambda+\abs{p}^2) +  2 p_1^2 h'(\lambda+\abs{p}^2), 2 p_1 p_2 h'(\lambda+\abs{p}^2)
\big)
\]
and hence
\begin{align}
\notag
\abs{\nabla \hat v(p)}^2
&\le
2 {h(\lambda+\abs{p}^2)}^2 + 8 |p|^4 {h'(\lambda+\abs{p}^2)}^2
\\
\label{nablav}
&\le
2 {h(\lambda+\abs{p}^2)}^2 + 8 (\lambda+\abs{p}^2)^2 {h'(\lambda+\abs{p}^2)}^2.
\end{align}
From the explicit form of $h$:
\[
\abs{h'(x)x}
=
\abs{
\frac{1-x^2} {x (1+x^2) \log^2\left(\frac{1}{x}+x\right)}-
\frac{1}{x \log\left(\frac{1}{x}+x\right)}
}
\le
2 \abs{\frac{1}{x \log\left(\frac{1}{x}+x\right)}}
=
2 \abs{h(x)}
\]
which together with \eqref{nablav} yields
\[
\abs{\nabla \hat v(p)}^2
\le
34 {h(\lambda + \abs{p}^2)}^2.
\]
Thus it is enough to bound
\[
\int_{\R^2}
\hat V(p) \abs{p}^2 \sup_{r:\abs{r-p}<\abs{p}/3} {h(\lambda+\abs{r}^2)}^2 dp.
\]
Note that from the definition of $h(x)$ (see \eqref{choice}) we have
\begin{equation}
\label{last}
\abs{p}^2 \sup_{r: \abs{r-p}\le \abs{p}/3} {h(\lambda+\abs{r}^2)}^2
\le
81 \frac{1}{\abs{p}^2 {\log^2(\abs{p}^2/9+9 \abs{p}^{-2})}}.
\end{equation}
The lemma now follows from the fact that $\hat V(p)$ is bounded and the function on the right hand side \eqref{last} is locally integrable near $\abs{p}=0$ in $\R^2$.
\end{proof}

\subsection{The non-isotropic SRBP}
\label{ss:nonisotrop}

The proof of the lower bound is considerably simpler in that case. We use the same strategy as before, setting
\[
D(\lambda,|p|)=4 \int_{\R^2} \hat V(p) \frac{1}{\lambda+|p-q|^2}dq
\]
by Schwarz's inequality we get
\[
J_3(\hat u)\le \int_{\R^2} \hat V(p) D(\lambda,|p|) p_1^2 \hat u(p)^2 dp.
\]
As in the DCGF case we have that if $\lambda<1$ and $|p|\le 1$ then
\[
D(\lambda,|p|)\le C \abs{\log(\lambda+p^2)}.
\]
Using the same arguments as before we get the following lower bound for sufficiently small $\lambda$:
\begin{eqnarray*}
(\varphi, R_\lambda \varphi)
&\ge&
\int_{\R^2} \ind{\abs{p}<1}
\hat V(p) \frac1{
\lambda + |p|^2+ C p_1^2 \abs{\log(\lambda + \abs{p}^2)}} dp.\\
&\ge& \int_{\R^2} \ind{\abs{p}<1/2}
\hat V(p) \frac1{
\lambda + |p|^2+ C p_1^2 \abs{\log(\lambda)}} dp.
\end{eqnarray*}
Changing to polar coordinates and using that $\hat V(p)>C'>0$ for $|p|$ small enough we have the lower bound
\begin{eqnarray*}
C'\int_0^\eps\int_0^{2\pi} \frac{r}{\lambda+r^2+C |\log\lambda| r^2 \cos^2 \alpha} d\alpha dr&=&C'\int_0^\eps \frac{2 \pi r }{\sqrt{\left(r^2+\lambda \right) \left(C |\log \lambda| r^2+r^2+\lambda
   \right)}}dr\\
&\ge & C'' \sqrt{|\log \lambda|}.
\end{eqnarray*}
This proves Theorem \ref{thm:masodik}.

\section{Appendix: The Alder-Wainwright scaling argument for superdiffusivity of tracer motion}
\label{s:appendix}

We reproduce the nonrigorous, nevertheless very instructive scaling argument due to B.J. Alder and T.E. Wainwright, respectively, to D. Forster, D. Nelson and  M. Stephen, which sheds sharp light on the origins of superdiffusivity in tracer particle motions. The original papers are \cite{alder_wainwright_67}, \cite{alder_wainwright_70}, \cite{forster_nelson_stephen_70}, see also \cite{spohn_91}.

\subsection{General formal setup and notation}
\label{ss:app_general}

Let $t\mapsto X(t)\in\R^d$ be a random motion of a tracer particle, with stationary and ergodic increments. The motion is performed in some random environment which also evolves in time. The time evolution of the tracer particle and that of the background environment may mutually influence one another.

Usually we decompose the random motion as sum of a martingale and its compensator:
\begin{equation}
\label{martingale+compensator}
X(t)=M(t) + \int_0^t V(s)ds,
\end{equation}
where $M(t)$ is a square integrable martingale with stationary and ergodic increments and the compensator $V(t)$ is a stationary and ergodic process which is also square integrable. We'll call  $V(t)$ the instantaneous velocity of the tracer. We are interested in understanding the origins of the possibly superdiffusive behaviour of $X(t)$. Since $M(t)$ is anyway diffusive, we don't care much about it. The main contribution to the superdiffusive behaviour anyway comes from the compensator (integral term) on the right hand side of \eqref{martingale+compensator}.

It is assumed that the instantaneous velocity comes from some background velocity field $U(t,x)$ as follows:
\begin{equation}
\label{velocity}
V(t)=U(t,X(t)).
\end{equation}
It is important to note that we think about a joint random dynamics $t\mapsto\big(X(t),U(t, \cdot)\big)$. The process of the velocity field $t\mapsto U(t,\cdot)$ is usually \emph{not stationary} while the instantaneous velocity of the tracer given in \eqref{velocity} is.

We will consider here only the \emph{isotropic} case. The results of similar considerations applied to non-isotropic cases will be also summarized at the end of this appendix.

The \emph{correlations of the velocity field}:
\[
K(t,x):=
\expect{ U(0,0)\cdot U(t,x) }.
\]
The \emph{velocity autocorrelation function}:
\[
C(t):=
\expect{V(0) \cdot V(t)}
\]
Mind that the velocity process $V(t)$ is assumed stationary (and ergodic), thus
\[
\expect{V(s) \cdot V(t)} = C(t-s).
\]
The \emph{variance} of the displacement is:
\[
E(t):=
\mathbf{E} \abs{X(t)}^2,
\qquad
\tilde E(t):=
\mathbf{E} \, {\abs{\int_0^t V(s)}^2}.
\]
Since the martingale part in \eqref{martingale+compensator} is anyway diffusive, in case of superdiffusive behaviour
\begin{align*}
E(t) \asymp \tilde E(t).
\end{align*}
By stationarity of $V(t)$
\begin{align*}
\label{green_kubo}
\tilde E(t)=2\int_0^t (t-s) C(s) ds,
\end{align*}
and hence
\begin{equation}
\label{green_kubo_asymp}
\tilde E(t)\asymp t \int_0^t C(s) ds.
\end{equation}
We shall refer to \eqref{green_kubo_asymp} as the \emph{asymptotic form of the Green-Kubo formula}.

\subsection{Scaling assumptions and computation}
\label{ss:app_scaling}

\medskip
\noindent
{\bf Scaling of the displacement:}
We assume that $X(t)$ is of order
\[
\alpha(t)=t^{\nu}(\log t)^\gamma.
\]
More explicitly, assume that
\begin{equation}
\label{scaling_X_iso}
\prob{X(t)\in dx} \asymp \alpha(t)^{-d}
\varphi(\alpha(t)^{-1} x) dx,
\end{equation}
where $\varphi:\R^d\to \R_+$ is a density, which is regular at $x=0$ and decays fast at $\abs{x}\to\infty$.

\medskip
\noindent
{\bf Scaling of the velocity field:}
We also assume that the correlations of the velocity field $U(t,x)$ scale as
\[
K(t,x) \asymp \beta(t)^{-d}\psi(\beta(t)^{-1} x).
\]
Note that under this assumption
\[
\int_{\R^d} K(t,x) dx \asymp \mathrm{const.}
\]
This corresponds to some kind of \emph{conservation of momentum} carried by the velocity field.

\medskip
\noindent
{\bf Main scaling assumptions:}
We make two important assumptions

\begin{enumerate}[1.]

\item
\emph{Regularity of $\psi$ and $\varphi$:}
The following regularity conditions hold:
\begin{equation}
\label{psi_L1_iso}
\int_{\R^d}\abs{\psi(x)} dx <\infty,
\qquad
\hat\psi(0) >0.
\end{equation}

\item
\emph{Displacement scales on faster order than the velocity-field correlations:}
\begin{equation}
\label{beta_less_than_alpha_iso}
\beta (t) = \Ordo (\alpha(t)).
\end{equation}
\end{enumerate}
Next we compute the velocity autocorrelation function:
\begin{align}
\notag
C(t)
&=
\expect{V(0)\cdot V(t)}
=
\expect{U(0,0)\cdot U(t,X(t))}
\\
\notag
&\approx
\int_{\R^d} K(t,x) \prob{X(t)\in dx}
&&
{\tt{decoupling}}
\\
\notag
&\asymp
\int_{\R^d} \beta(t)^{-d}\psi(\beta(t)^{-1}x)  \alpha(t)^{-d}\varphi(\alpha(t)^{-1}x) dx
&&
{\tt{scaling}}
\\
\notag
&=
\alpha(t)^{-d}\int_{\R^d} \psi(x)\varphi(\beta(t) \alpha(t)^{-1} x) dx
\\
\label{compute_C(t)_iso}
&\asymp
\alpha(t)^{-d}
\end{align}
Of course, the second step (decoupling) is the shaky one. In the last step we used the main scaling asumptions \eqref{psi_L1_iso} and \eqref{beta_less_than_alpha_iso}. Regularity of $\varphi$ at $x=0$ was also assumed.

\subsection{Conclusions}
\label{ss:app_conclusions}

From the scaling assumption \eqref{scaling_X_iso} it follows that
\begin{equation}
\label{scaling_tilde_D_iso}
\tilde E(t) \asymp \alpha(t)^2.
\end{equation}
On the other hand, using the Green-Kubo formula \eqref{green_kubo_asymp} and the computations in \eqref{compute_C(t)_iso} we get
\begin{equation}
\label{tilde_D_asymp_from_gk_iso}
\tilde E(t)\asymp t \int^t \alpha(s)^{-d} ds.
\end{equation}
The only choices of the scaling function $\alpha(t)$ consistent with both \eqref{scaling_tilde_D_iso} and \eqref{tilde_D_asymp_from_gk_iso}, are
\begin{align*}
\mathbf{d=1:}
&&&\qquad
\nu=\frac23, \quad \gamma=0,
&&\qquad
E(t)\asymp t^{4/3},
\\
\mathbf{d=2:}
&&&\qquad
\nu=\frac12, \quad \gamma=\frac14,
&&\qquad
E(t)\asymp t(\log t)^{1/2},
\\
\mathbf{d=3:}
&&&\qquad
\nu=\frac12, \quad \gamma=0,
&&\qquad
E(t)\asymp t.
\end{align*}

If $d=2$ and the system is non-isotropic in the sense that $X_2(t)$ is assumed diffusive and $X_1(t)$ possibly superdiffusive, very similar considerations and computations lead to the following conclusion:
\begin{align*}
\mathbf{d=2:}
&&&\qquad
\nu=\frac12, \quad \gamma=\frac13,
&&\qquad
E(t)\asymp t (\log t)^{2/3}.
\end{align*}

\bigskip

\begin{ack}
The work of BT was partially supported by OTKA (Hungarian National Research Fund) grant K 60708. BV was partially supported by the National Science Foundation grant DMS-0905820. BV thanks J. Quastel for many stimulating conversations.
\end{ack}

\vskip4cm

\noindent
{\sc Addresses of authors:}
\\[15pt]
BT:
Institute of Mathematics, Budapest University of Technology, Egry J\'ozsef u.\ 1, Budapest 1111, Hungary,
 {\tt balint@math.bme.hu}
\\[15pt]
BV:
Department of Mathematics, University of Wisconsin -- Madison, 480 Lincoln Drive, Madison WI 53706, USA, {\tt valko@math.wisc.edu}


\begin{thebibliography}{99}

\bibitem{alder_wainwright_67}
B.J. Alder, T.E. Wainwright:
Velocity autocorrelation for hard spheres.
{\sl Phys. Rev. Letters}, {\bf 18}: 988-990 (1967)


\bibitem{alder_wainwright_70}
B.J. Alder, T.E. Wainwright:
Decay of the velocity autocorrelation function.
{\sl Phys. Rev. A} {\bf 1}: 18-21 (1970)


\bibitem{amit_parisi_peliti_83}
D. Amit, G. Parisi, L. Peliti:
Asymptotic behavior of the `true' self-avoiding walk.
{\sl Physical Reviews B} {\bf 27}: 1635--1645 (1983)


\bibitem{durrett_rogers_92}
R.T. Durrett, L.C.G.Rogers:
Asymptotic behavior of Brownian polymers.
{\sl Probab. Theory Rel. Fields} {\bf 92}: 337--349 (1992)

\bibitem{forster_nelson_stephen_70}
D. Forster, D. Nelson, M. Stephen:
Large distance and long time properties of a randomstirred fluid.
{\sl Phys. Rev. A} {\bf 16}: 732-749 (1970)

\bibitem{horvath_toth_veto_10}
I. Horv\'ath, B. T\'oth, B. Vet\H o:
Diffusive limits for "true" (or myopic) self-avoiding random walks and self-repellent Brownian polymers in three and more dimensions.
{\tt http://arxiv.org/abs/1009.0401} (submitted, 2010)

\bibitem{janson_97}
S. Janson:
{\sl Gaussian Hilbert Spaces.}
Cambridge University Press, 1997

\bibitem{komorowski_landim_olla_2011}
T. Komorowski, C. Landim, S. Olla:
{\sl Fluctuations in Markov Processes: Time Symmetry and Martingale Approximation},
Grundlehren der Mathematischen Wissenschaften, Springer Verlag, 2011 (to appear)


\bibitem{komorowski_olla_02}
T. Komorowski, S. Olla:
On the superdiffusive behaviour of passive tracer with a Gaussian drift.
{\sl Journal of Statistical Physics} {\bf 108}: 647--668 (2002)

\bibitem{kozlov_1985}
S.M. Kozlov:
The method of averaging and walks in inhomogeneous environments.
{\sl Uspekhi Mat. Nauk} {\bf 40}: 61-120 (1985)
[English translation:
{\sl Russian Math. Surveys} {\bf 40}: 73-145 (1985)]

\bibitem{landim_ramirez_yau}

C. Landim, A. Ramirez, H-T. Yau:
Superdiffusivity of Two Dimensional Lattice Gas Models. \textit{Journal of Statistical Physics},
\textbf{119}, Numbers 5-6, 963-995 (2005)




\bibitem{landim_quastel_salmhofer_yau_04}
C. Landim, J. Quastel, M. Salmhofer, H-T. Yau:
Superdiffusivity of one and two dimentsional asymmetric simple exclusion processes.
{\sl Communicarions in Mathematical Physics} {\bf 244}: 455--481 (2004)


\bibitem{norris_rogers_williams_87}
J.R. Norris, L.C.G. Rogers, D. Williams:
Self-avoiding walk: a
Brownian motion model with local time drift.
{\sl Probab. Theory Rel. Fields} {\bf 74}: 271--287 (1987)

\bibitem{obukhov_peliti_83}
S.P. Obukhov, L. Peliti:
Renormalisation of the ``true'' self-avoiding walk.
{\sl Journal of  Physics A} {\bf 16}: L147--L151 (1983)

\bibitem{peliti_pietronero_87}
L. Peliti, L. Pietronero:
Random walks with memory.
{\sl Rivista del Nuovo Cimento} {\bf 10}: 1--33 (1987)

\bibitem{reed_simon_vol1_80}
M. Reed, B. Simon:
{\sl Methods of Modern Mathematical Physics Vol 1, 2.}
Academic Press New York, 1972--75.

\bibitem{simon_74}
B. Simon:
{\sl The $P(\phi)_2$ Euclidean (Quantum) Field Theory.}
Princeton University Press, 1974.

\bibitem{spohn_91}
H. Spohn:
{\sl Large Scale Dynamics of Interacting Particles.}
Springer-Verlag, Berlin-Heidelberg-New York, 1991

\bibitem{tarres_toth_valko_09}
P. Tarr\`es, B. T\'oth, B. Valk\'o:
Diffusivity bounds for 1d Brownian polymers.
{\sl Ann. Probab.} (to appear)
{\tt http://arxiv.org/abs/0911.2356}

\bibitem{toth_95}
B. T\'oth:
`True' self-avoiding walk with bond repulsion on $\Z$: limit theorems.
{\sl Ann. Probab.}, {\bf 23}: 1523-1556 (1995)

\bibitem{toth_01}
B. T\'oth:
Self-interacting random motions.
In: {\sl Proceedings of the 3rd European Congress of Mathematics}, Barcelona 2000, vol. 1, pp. 555-565, Birkhauser, 2001.

\bibitem{toth_veto_09}
B.\ T\'oth, B.\ Vet{\H o}:
Continuous time `true' self-avoiding random walk on $\Z$.
{\sl ALEA -- Latin American Journal of Probability} (2010, to appear) {\tt http://arxiv.org/abs/0909.3863}


\bibitem{toth_werner_98}
B. T\'oth, W. Werner:
The true self-repelling motion.
{\sl Probab. Theory Rel. Fields}, {\bf 111}: 375-452 (1998)



\bibitem{yau_04}
H-T. Yau: $(\log t)^{2/3}$ law of the two dimensional asymmetric simple exclusion process.
{\sl Annals of Mathematics} {\bf 159}: 377--105 (2004)
\end{thebibliography}
\end{document}